\documentclass[EJP,preprint]{ejpecp} % replace ECP by EJP if needed.
\SHORTTITLE{Normalized Laplacian spectrum in preferential attachment}

\TITLE{Limiting empirical spectral measure of the normalized Laplacian in preferential attachment graphs.}%\
\newcommand{\Cp}{\mathbb{C}_+}
\newcommand{\C}{\mathbb{C}}
\newcommand{\N}{\mathbb{N}}
\newcommand{\R}{\mathbb{R}}
\newcommand{\E}{\mathbb{E}}
\newcommand{\K}{\mathcal{E}}
\newcommand{\PP}{\mathbb{P}}
\newcommand{\Tr}{{\bf Tr}}

\newcommand{\dist }{\text{dist}}

\AUTHORS{%
  Malika~Kharouf\footnote{University of Technology of Troyes, France.
    \EMAIL{malika.kharouf@utt.fr}}}%\orcid{0000-0009-0005-3218-9594}
 % Separate items with ;

\AMSSUBJ{05C80} % Edit. Separate items with ;
\SUBMITTED{January 2, 2013} % Edit.
\ACCEPTED{December 13, 2014} % Edit.

\ABSTRACT{We study the empirical spectral distribution of the normalized Laplacian of linear preferential attachment graphs in the Barab\'asi-Albert regime with fixed out-degree. For the resulting sequence of random multigraphs, we prove that the empirical spectral distribution converges weakly in probability to a deterministic probability measure supported on the interval $[0,2]$. The limit is characterized via the local weak limit of preferential attachment graphs (the P\'olya--point graph): the limiting Stieltjes transform is given by the expected diagonal Green function at the root of the normalized Laplacian operator on this infinite random graph. The proof combines a resolvent approach with a uniform Neumann-series expansion for the normalized Laplacian, a random-walk representation in terms of return probabilities on decorated neighborhoods, a truncation and Doob martingale--Azuma--Hoeffding concentration argument along the PA filtration,
and an analytic continuation argument based on normal families.
}

\begin{document}

%%%%%%%%%%%%%%%%%%%%%%%%%%%%%%%%%%%%%%%%%%%%%%%%%%%%%%%%%%%%%%%%%%%
%%                                                               %%
%% No need for \maketitle.                                       %%
%%                                                               %%
%%%%%%%%%%%%%%%%%%%%%%%%%%%%%%%%%%%%%%%%%%%%%%%%%%%%%%%%%%%%%%%%%%%

%%%%%%%%%%%%%%%%%%%%%%%%%%%%%%%%%%%%%%%%%%%%%%%%%%%%%%%%%%%%%%%%%%%
%%                                                               %%
%% Please replace what follows by the body of your article       %%
%% (up to the bibliography):                                     %%
%%                                                               %%
%%%%%%%%%%%%%%%%%%%%%%%%%%%%%%%%%%%%%%%%%%%%%%%%%%%%%%%%%%%%%%%%%%%

\section{Introduction}

\subsection{Motivation: spectra of the normalized Laplacian in sparse, heterogeneous networks}

Spectral statistics of graph operators are a basic language for random networks: they encode
geometry and transport (mixing, diffusion, expansion) and provide tractable proxies for dynamical
processes. Among the canonical operators, the symmetric normalized Laplacian
\[
L := I - D^{-1/2} A D^{-1/2}
\]
is particularly well-suited to sparse and highly inhomogeneous graphs: its spectrum is always
contained in $[0,2]$, it is tightly linked to the simple random walk, and it governs classical
functional inequalities and mixing estimates; see, e.g., \cite{Chung1997,SpielmanNotes,LevinPeres2017}.
This paper establishes a deterministic limiting law for the empirical spectral distribution (ESD)
of $L$ in the Barab\'asi--Albert preferential attachment regime.

\subsection{Preferential attachment: scale-free growth and correlations}

Preferential attachment (PA) is a growth mechanism producing heavy-tailed degree
distributions: vertices arrive sequentially and connect to existing ones with probability proportional
to current degree \cite{BarabasiAlbert1999}. Rigorous analysis of the degree sequence and its
asymptotic power law goes back to \cite{BollobasRiordanSpencerTusnady2001}; see also the
monograph \cite{VanDerHofstad2017} and references therein. Numerous variants
(sublinear attachment, fitness, geometry, \dots) exhibit different universality classes and phase
transitions, see, e.g., \cite{Dereich2016Fitness,DereichMorters2009,Jordan2013}.

From the spectral viewpoint, PA graphs pose two intertwined difficulties. First, degrees are
unbounded and strongly heterogeneous (hubs coexist with typical vertices). Second, the growth
mechanism creates pronounced temporal correlations (degree--age dependence and nontrivial
dependencies between edges), placing PA outside independent-edge frameworks such as
expected-degree models. These features make it delicate to transfer the now-standard
``local weak limit $\Rightarrow$ limiting spectrum'' paradigm from bounded-degree graphs.

\subsection{Local weak limits and the P\'olya--point graph}

A central organizing principle for sparse random graphs is local weak convergence
(Benjamini--Schramm convergence): one samples a uniform vertex, inspects its neighborhood of
fixed radius, and lets the graph size diverge \cite{BenjaminiSchramm2001,AldousLyons2007}.
For preferential attachment, the typical neighborhood retains ``age'' information and is far from
exchangeable. A key result  of Berger--Borgs--Chayes--Saberi \cite{BergerBorgsChayesSaberi2014}
identifies the local weak limit for a broad class of linear PA graphs as an explicit marked infinite
random graph, often called the P\'olya--point graph. This limit has since become a key
approximation device for local network properties (e.g.\ robustness/vulnerability phenomena,
see \cite{EckhoffMorters2014}).

\subsection{Spectral limits for sparse graphs}

For bounded-degree, locally tree-like graphs, resolvent/Green function methods relate limiting ESDs
of adjacency-type operators to the root resolvent on the local weak limit; a systematic approach
appears in \cite{BordenaveLelarge2010}. For random trees, convergence of ESDs under fringe/local
assumptions is well understood, and preferential attachment trees fall within that scope for
adjacency-type operators \cite{BhamidiEvansSen2012}.

In contrast, for power-law graphs with independent edges (Chung--Lu type models),
Chung--Lu--Vu \cite{ChungLuVu2003} show that the spectral behavior depends sensitively on the
operator: adjacency spectra display heavy-tail effects, while the normalized Laplacian can exhibit
random-matrix-type behavior under suitable degree conditions. Related Laplacian spectral limit results in scale-free models with conditionally independent edges include Hazra and Malhotra \cite{HazraMalhotra2025} who study the empirical spectral distribution of a centred Laplacian for scale-free percolation  and identify an explicit limiting law in one of their main parameter regimes. Preferential attachment graphs share
the power-law degree feature, but their temporal correlations lead to markedly different spectral
phenomena. For instance, for adjacency matrices, extreme eigenvalues are driven by hubs and leading
eigenvectors may be highly localized; see \cite{DorogovtsevGoltsevMendesSamukhin2003,MontealegreVu2019}
and references therein. Providing a  limit theorem for the ESD of the normalized Laplacian
in the Barab\'asi--Albert preferential attachment regime is the goal of the present paper.

\subsection{Our contribution}

We consider the  preferential attachment multigraph in which each new vertex introduces $m\ge 2$
edges and endpoints are chosen proportionally to degree (sampling with replacement).
Let $L_n$ be the normalized Laplacian on the $n$-vertex graph and let
\[
\mu_n := \frac{1}{n}\sum_{i=1}^n \delta_{\lambda_i(L_n)}
\]
be its empirical spectral distribution. Our main result proves that $\mu_n$ converges weakly in
probability to a \emph{deterministic} probability measure $\mu$ supported on $[0,2]$.

The limit is identified through the P\'olya--point local weak limit $(G_\infty,o)$ of
\cite{BergerBorgsChayesSaberi2014}. Writing $L_\infty := I - D_\infty^{-1/2}A_\infty D_\infty^{-1/2}$
for the normalized Laplacian operator on $G_\infty$, we show that for every $z\in\mathbb{C}_+$,
the Stieltjes transform
\[
m_n(z) \;:=\; \int \frac{1}{x-z}\,\mu_n(dx) \;=\; \frac{1}{n}\mathrm{Tr}(L_n-zI)^{-1}
\]
converges in probability to the deterministic Herglotz function
\[
m(z) \;=\; \mathbb{E}\big\langle \delta_o,(L_\infty-zI)^{-1}\delta_o\big\rangle,
\]
and $\mu$ is the unique probability measure with Stieltjes transform $m$
(\cite{BaiSilverstein2010}).

\subsection{Organization}

Section~\ref{sec:model} defines the preferential attachment multigraph and the operators.
Section~\ref{sec:neumann} develops the resolvent bounds and the Neumann expansion with uniform error control.
Section~\ref{sec:locality} identifies diagonal powers of $W_n$ with random walk return probabilities and
formalizes their locality as functionals of decorated rooted balls. Section~\ref{sec:LWC_concentration} presents the local weak
convergence to the P\'olya--point graph and concentration of local averages. Section  ~\ref{sec:neumann-domain} gives the convergence on the
Neumann domain, and Section~\ref{sec:extension} extends convergence to all of $\mathbb{C}_+$ and concludes the
proof of $\mu_n\Rightarrow\mu$ in probability.

\section{Model, operators and main result }\label{sec:model}

\subsection{The Barab\'asi--Albert preferential attachment multigraph}\label{subsec:PA-model}

Fix an integer $m\ge 2$.
For each $n\ge 2$, we construct a random (multi)graph $G_n$ on vertex set
\[
V(G_n)=\{1,2,\dots,n\},
\]
where the label of a vertex coincides with its birth time.

\paragraph{Initial condition.}
We start with $G_2$ consisting of two vertices $\{1,2\}$ joined by $m$ parallel edges.
Equivalently, the (multi-)degree satisfies $d_2(1)=d_2(2)=m$.

\paragraph{Growth step.}
Given $G_{n-1}$ for some $n\ge 3$, we obtain $G_n$ by adding the new vertex $n$ together with
$m$ edges connecting $n$ to vertices in $\{1,\dots,n-1\}$.
More precisely, conditionally on $G_{n-1}$ we sample endpoints
\[
W_{n,1},\dots,W_{n,m}\in\{1,\dots,n-1\}
\]
\emph{independently} with
\begin{equation}\label{eq:PA-rule}
\mathbb{P}(W_{n,i}=v \mid G_{n-1})
=\frac{d_{n-1}(v)}{\sum_{u=1}^{n-1} d_{n-1}(u)}
=\frac{d_{n-1}(v)}{2m(n-2)},
\qquad v\in\{1,\dots,n-1\}.
\end{equation}
We then add edges $(n,W_{n,1}),\dots,(n,W_{n,m})$ to $G_{n-1}$.
Sampling with replacement allows multiple edges between $n$ and a given older vertex.
We exclude self-loops since the endpoints are chosen among $\{1,\dots,n-1\}$.

\paragraph{Degrees and total degree.}
Let $d_n(v)$ denote the (multi-)degree of $v$ in $G_n$.
Since each vertex is born with $m$ incident edges, we have $d_n(v)\ge m$ for all $v\le n$.
Moreover, the number of edges in $G_n$ is $m(n-1)$ and hence
\[
\sum_{v=1}^n d_n(v)=2m(n-1).
\]

\subsection{Adjacency, random-walk kernel, and normalized Laplacian}\label{subsec:operators}

Let $A_n$ be the $n\times n$ adjacency matrix of $G_n$ with multiplicities:
for $u\neq v$, $(A_n)_{uv}$ is the number of edges between $u$ and $v$, and $(A_n)_{uu}=0$.
Let $D_n=\mathrm{diag}(d_n(1),\dots,d_n(n))$ be the diagonal degree matrix.

\paragraph{Simple random walk and its kernel.}
We consider the (discrete-time) simple random walk on the multigraph $G_n$.
Given a current vertex $u\in\{1,\dots,n\}$, the walk moves in one step to a neighbor $v$ by choosing
uniformly among the $d_n(u)$ incident half-edges at $u$.  In particular, edge multiplicities are taken into account.
Equivalently, the transition probabilities are
\begin{equation}\label{eq:rw-kernel}
  \mathbb{P}(X_{t+1}=v \mid X_t=u,\,G_n)
  \;=\;
  \frac{(A_n)_{uv}}{d_n(u)},
  \qquad u,v\in\{1,\dots,n\}.
\end{equation}
In matrix form, the transition kernel is the row-stochastic matrix
\[
  P_n \coloneqq D_n^{-1}A_n .
\]
The chain is reversible with stationary distribution $\pi_n(u)=d_n(u)/\sum_{w}d_n(w)$ (see, e.g.,
\cite[Ch.~2]{LevinPeres2017} or \cite[Ch.~1]{Chung1997}).
For $k\ge 1$, the diagonal entry $(P_n^k)_{uu}$ equals the $k$-step return probability to $u$.

\paragraph{Normalized adjacency and normalized Laplacian.}
Define the symmetric normalized adjacency matrix and normalized Laplacian by
\begin{equation}\label{eq:norm-adj-lap}
W_n \coloneqq D_n^{-1/2}A_nD_n^{-1/2},
\qquad
L_n \coloneqq I_n - W_n.
\end{equation}
The operator $W_n$ is similar to $P_n$ via
\begin{equation}\label{eq:similarity}
W_n = D_n^{1/2} P_n D_n^{-1/2},
\end{equation}
and therefore they have the same spectrum.
Moreover, for every $k\ge 1$ and every vertex $v$,
\begin{equation}\label{eq:diag-power-return}
(W_n^k)_{vv} = (P_n^k)_{vv}.
\end{equation}
\begin{remark}
$\mathcal{L}_n$ is symmetric positive semidefinite.
Moreover, all eigenvalues of $\mathcal{L}_n$ lie in $[0,2]$ \cite{Chung1997,SpielmanNotes}.
\end{remark}

\subsection{Empirical spectral distribution and Stieltjes transform}\label{subsec:ESD}

Let $\lambda_1^{(n)},\dots,\lambda_n^{(n)}$ be the eigenvalues of $L_n$ (counted with multiplicity).
The empirical spectral distribution (ESD) of $L_n$ is the random probability measure
\begin{equation}\label{eq:ESD}
\mu_n \coloneqq \frac{1}{n}\sum_{i=1}^n \delta_{\lambda_i^{(n)}}.
\end{equation}
For $z\in\mathbb{C}_+\coloneqq\{z\in\mathbb{C}:\Im z>0\}$, define the resolvent and the Stieltjes transform
\begin{equation}\label{eq:resolvent-stieltjes}
G_n(z) \coloneqq (L_n-zI_n)^{-1},
\qquad
m_n(z) \coloneqq \int_{\mathbb{R}} \frac{1}{x-z}\,\mu_n(dx)
= \frac{1}{n}\mathrm{Tr}\,G_n(z).
\end{equation}
\begin{remark}[Uniform resolvent bound \cite{ReedSimon1980,Bhatia1997}]
Since $\mathrm{Spec}(\mathcal{L}_n)\subset[0,2]$, for all $z\in\mathbb{C}^+$ we have
\[
\|G_n(z)\|\le \frac{1}{\Im z}
\quad\Rightarrow\quad |(G_n(z))_{uu}|\le \frac{1}{\Im z}.
\]
\end{remark}

\subsection{The local weak limit and the limiting operator}\label{subsec:lwc-limit-operator}

Let $U_n$ be a uniformly chosen vertex of $G_n$, conditionally independent given $G_n$.
We view $(G_n,U_n)$ as a random rooted multigraph (with edge multiplicities) and consider
the \emph{local} topology on rooted graphs: for $r\ge 0$, the rooted $r$-ball
$B_r(G_n,U_n)$ is the rooted subgraph induced by vertices at graph distance at most $r$ from $U_n$.

\begin{theorem}[Local weak limit \cite{BergerBorgsChayesSaberi2014}]\label{thm:lwc}
As $n\to\infty$, the rooted graphs $(G_n,U_n)$ converge in distribution for the local topology to a random
infinite rooted graph $(G_\infty,o)$.
The law of $(G_\infty,o)$ is the \emph{P\'olya--point graph} constructed in
\cite{BergerBorgsChayesSaberi2014}; see also \cite{BenjaminiSchramm2001,AldousLyons2007} for the general framework.
\end{theorem}

We next define the limiting operators associated with $(G_\infty,o)$.
Write $A_\infty$ for the adjacency matrix of $G_\infty$ (with $(A_\infty)_{xx}=0$ and
$(A_\infty)_{xy}$ equal to the number of edges between $x\neq y$), and let
$d_\infty(x)\coloneqq \sum_{y}(A_\infty)_{xy}$ be the degree of $x$.
Almost surely, $d_\infty(x)<\infty$ for every vertex $x$ \cite{BergerBorgsChayesSaberi2014}.

Define the normalized adjacency operator $W_\infty$ by
\begin{equation}\label{eq:Winfty}
  (W_\infty)_{xy}\;\coloneqq\;
  \begin{cases}
    \displaystyle \frac{(A_\infty)_{xy}}{\sqrt{d_\infty(x)d_\infty(y)}}, & x\neq y,\\[1ex]
    0, & x=y,
  \end{cases}
\end{equation}
and the normalized Laplacian operator
\begin{equation}\label{eq:Linfty}
  L_\infty \;\coloneqq\; I - W_\infty .
\end{equation}
Since $G_\infty$ is locally finite a.s., the operator $W_\infty$ defined in \eqref{eq:Winfty} is a bounded
self-adjoint operator on $\ell^2(V(G_\infty))$ with $\|W_\infty\|\le 1$; consequently
$L_\infty=I-W_\infty$ is bounded self-adjoint with $\mathrm{Spec}(L_\infty)\subset[0,2]$
(see, e.g., \cite[Ch.~1]{Chung1997} and \cite{BordenaveLelarge2010}).
Therefore, for every $z\in\mathbb C_+$ the resolvent $(L_\infty-zI)^{-1}$ exists and satisfies
$\|(L_\infty-zI)^{-1}\|\le (\Im z)^{-1}$.

We denote the rooted Green function by
\begin{equation}\label{eq:root-green}
  g_\infty(o;z)\;\coloneqq\;\big\langle \delta_o,\,(L_\infty-zI)^{-1}\delta_o\big\rangle .
\end{equation}
Here $\delta_o\in \ell^2(V(G_\infty))$ denotes the canonical basis vector at the root, i.e.\ $\delta_o(x)=\mathbf{1}_{\{x=o\}}$.

\subsection{Main result and roadmap}\label{subsec:main-roadmap}

Let $\mu_n$ be the empirical spectral distribution of $L_n$ defined in \eqref{eq:ESD}, and let
$m_n(z)=\int (x-z)^{-1}\mu_n(dx)=\frac1n\mathrm{Tr}(L_n-zI)^{-1}$ be its Stieltjes transform.
Define the candidate limiting Stieltjes transform by
\begin{equation}\label{eq:limit-stieltjes}
  m(z)\;\coloneqq\;\mathbb{E}\big[g_\infty(o;z)\big], \qquad z\in\mathbb{C}_+,
\end{equation}
where $g_\infty$ is given by \eqref{eq:root-green}. By the spectral theorem for bounded self-adjoint operators (see, e.g., \cite{ReedSimon1980}),
for each realization of $(G_\infty,o)$ there exists a probability measure $\mu_{\infty,o}$
supported on $\mathrm{Spec}(L_\infty)\subset[0,2]$ such that
$g_\infty(o;z)=\int (\lambda-z)^{-1}\,\mu_{\infty,o}(d\lambda)$ for all $z\in\mathbb C_+$.

\begin{theorem}[Limiting empirical spectral measure]\label{thm:main}
Fix $m\ge 2$ and let $(G_n)_{n\ge2}$ be the Barab\'asi--Albert preferential attachment multigraph defined in
Section~\ref{subsec:PA-model}. Let $L_n$ be the normalized Laplacian \eqref{eq:norm-adj-lap} and $\mu_n$ its ESD.
\begin{enumerate}
\item For every $z\in\mathbb{C}_+$,
\[
  m_n(z)\xrightarrow[n\to\infty]{\mathbb{P}} m(z)=\mathbb{E}\big[g_\infty(o;z)\big].
\]
\item There exists a deterministic probability measure $\mu$ on $[0,2]$ such that
\[
  \mu_n \Rightarrow \mu \quad \text{weakly in probability},
\]
and $m$ is the Stieltjes transform of $\mu$.
\end{enumerate}
\end{theorem}

\paragraph{Roadmap of the proof.}
The argument follows the general ``local weak limit $\Rightarrow$ spectral limit'' paradigm for sparse graphs
and proceeds in five steps.

\smallskip\noindent
\textbf{Step 1: Neumann expansion on a nontrivial domain.}
Write $L_n-zI=(1-z)I-W_n$ with $W_n=D_n^{-1/2}A_nD_n^{-1/2}$.
Since $\|W_n\|\le 1$ and $\mathrm{Spec}(L_n)\subset[0,2]$, for $z$ in the domain
$\{z\in\mathbb{C}_+:\ |1-z|>1\}$ we expand the resolvent by a Neumann series and obtain a uniform truncation bound.
This reduces $m_n(z)$ to finitely many averages of diagonal terms $(W_n^k)_{uu}$.

\smallskip\noindent
\textbf{Step 2: Diagonal powers as return probabilities.}
Using the similarity between $W_n$ and the random-walk kernel $P_n=D_n^{-1}A_n$,
we identify $(W_n^k)_{uu}=(P_n^k)_{uu}$.
Hence $(W_n^k)_{uu}$ is the $k$-step return probability of the simple random walk started at $u$.

\smallskip\noindent
\textbf{Step 3: Locality.}
For $k\ge 0$ and a vertex $u\in V(G_n)$, let $B_k(G_n,u)$ denote the rooted radius-$k$ neighborhood of $u$,
i.e.\ the rooted subgraph induced by vertices at graph distance at most $k$ from $u$.
For each fixed $k$, the return probability $(P_n^k)_{uu}$ depends only on $B_k(G_n,u)$ together with the local
degree/edge-multiplicity information appearing in $A_n$ and $D_n$.
Thus $(W_n^k)_{uu}$ can be written as a bounded local functional of the rooted $k$-ball.

\smallskip\noindent
\textbf{Step 4: Local weak limit and  martingale concentration of local averages.}
After lifting local convergence to marked neighborhoods, we prove a self-averaging law for empirical local statistics by truncating high-degree neighborhoods, applying a Doob martingale-Azuma-Hoeffding concentration argument along the PA growth filtration, and then removing the truncation using the local weak limit. 

\smallskip\noindent
\textbf{Step 5: Extension to all $z\in\mathbb{C}_+$.}
The uniform bound $|m_n(z)|\le (\Im z)^{-1}$ implies that $(m_n)_n$ is a normal family on $\mathbb{C}_+$.
A Vitali/Montel argument upgrades convergence from the Neumann domain to all of $\mathbb{C}_+$.
Finally, the Stieltjes continuity theorem implies weak convergence of $\mu_n$ to a deterministic limit $\mu$.

\section{Resolvent bounds and a Neumann expansion}\label{sec:neumann}

\subsection{Resolvent preliminaries}\label{subsec:resolvent-prelim}

Recall that $L_n=I-W_n$ is a real symmetric $n\times n$ matrix and hence defines a bounded self-adjoint
operator on $\ell^2(\{1,\dots,n\})$. In particular, its spectrum is real and contained in $[0,2]$
(for normalized Laplacians, see, e.g., \cite[Ch.~1]{Chung1997}).
For $z\in\C\setminus\R$ we denote the resolvent by
\[
  G_n(z)\;\coloneqq\; (L_n-zI)^{-1},
\]
and for a vertex $u\in\{1,\dots,n\}$ we write the diagonal Green function
\[
  g_n(u;z)\;\coloneqq\;\langle \delta_u,\,G_n(z)\delta_u\rangle \;=\; (G_n(z))_{uu}.
\]
The Stieltjes transform of the empirical spectral distribution $\mu_n$ is
$m_n(z)=\frac1n\mathrm{Tr}\,G_n(z)=\frac1n\sum_{u=1}^n g_n(u;z)$, for $z\in\Cp$.

\begin{lemma}[Uniform resolvent bound]\label{lem:resolvent-bound}
For every $n\ge 2$ and every $z\in\Cp$,
\begin{equation}\label{eq:resolvent-bound}
  \|G_n(z)\|\;\le\; \frac{1}{\Im z}.
\end{equation}
Consequently, for every $u\in\{1,\dots,n\}$,
\begin{equation}\label{eq:diag-bound}
  |g_n(u;z)| \le \frac{1}{\Im z},
  \qquad
  |m_n(z)| \le \frac{1}{\Im z}.
\end{equation}
\end{lemma}

\begin{proof}
Since $L_n$ is real symmetric, there exists an orthonormal basis of eigenvectors and we may write
$L_n = U \Lambda U^\top$ with $\Lambda=\mathrm{diag}(\lambda_1^{(n)},\dots,\lambda_n^{(n)})$.
Hence
\[
  G_n(z)=U(\Lambda-zI)^{-1}U^\top,
\]
so $\|G_n(z)\|=\max_{1\le i\le n}|(\lambda_i^{(n)}-z)^{-1}| = \mathrm{dist}\big(z,\mathrm{Spec}(L_n)\big)^{-1}$.
For $z\in\Cp$ and $\lambda\in\R$, we have $|\lambda-z|\ge \Im z$, which yields \eqref{eq:resolvent-bound}.
The bounds in \eqref{eq:diag-bound} follow from $|g_n(u;z)|\le \|G_n(z)\|$ and
$|m_n(z)|\le \frac1n\sum_u |g_n(u;z)|\le \|G_n(z)\|$.
\end{proof}

\begin{lemma}[Herglotz property]\label{lem:herglotz}
For each $n\ge 2$ and each $u\in\{1,\dots,n\}$, the map $z\mapsto g_n(u;z)$ is analytic on $\Cp$ and satisfies
$\Im g_n(u;z)>0$ for $z\in\Cp$. In particular, $m_n$ is analytic on $\Cp$ and $\Im m_n(z)>0$ for $z\in\Cp$.
\end{lemma}

\begin{proof}
Analyticity follows because $z\mapsto (L_n-zI)^{-1}$ is analytic on $\C\setminus\mathrm{Spec}(L_n)$.
To see the sign of the imaginary part, note that for any vector $x\in\C^n$,
\[
  \Im \langle x,\, (L_n-zI)^{-1}x\rangle
  \;=\; (\Im z)\,\|(L_n-zI)^{-1}x\|^2 \;>\;0,
\]
which is a standard resolvent identity for self-adjoint operators; see, e.g., \cite[Ch.~VII.3]{ReedSimon1980}.
Taking $x=\delta_u$ gives $\Im g_n(u;z)>0$, and averaging over $u$ yields $\Im m_n(z)>0$.
\end{proof}

\subsection{Neumann expansion on the domain \texorpdfstring{$|1-z|>1$}{|1-z|>1}}\label{subsec:neumann-expansion}

Throughout this subsection we work on the open set
\begin{equation}\label{eq:neumann-domain}
  \mathcal{D}\;\coloneqq\;\{z\in\C_+:\ |1-z|>1\}.
\end{equation}
Note that $\mathcal{D}$ is nonempty and has an accumulation point,
which will later allow analytic continuation to all of $\C_+$.

\paragraph{A contraction estimate for $W_n$.}
Recall $W_n=D_n^{-1/2}A_nD_n^{-1/2}$ and $L_n=I-W_n$.
For normalized Laplacians one has $\mathrm{Spec}(L_n)\subset[0,2]$,
hence $\mathrm{Spec}(W_n)=\{1-\lambda:\lambda\in\mathrm{Spec}(L_n)\}\subset[-1,1]$.
Since $W_n$ is real symmetric, its operator norm equals its spectral radius, and therefore
\begin{equation}\label{eq:Wn-norm}
  \|W_n\|\le 1.
\end{equation}

\paragraph{Neumann series for bounded operators \cite{ReedSimon1980}}
We will repeatedly use the following standard fact.

\begin{lemma}[Neumann series and tail bound]\label{lem:neumann-general}
Let $T$ be a bounded linear operator on a Hilbert space. If $\|T\|<1$, then $I-T$ is invertible and
\begin{equation}\label{eq:neumann-general}
  (I-T)^{-1}=\sum_{k=0}^{\infty} T^k
\end{equation}
with convergence in operator norm. Moreover, for every integer $K\ge 1$,
\begin{equation}\label{eq:neumann-tail}
  \Big\|(I-T)^{-1}-\sum_{k=0}^{K-1}T^k\Big\|
  \le \frac{\|T\|^K}{1-\|T\|}.
\end{equation}
\end{lemma}

\begin{proof}
This is a classical result; see, e.g., \cite[Ch.~VI.1]{ReedSimon1980}.
For completeness, note that the partial sums $S_K=\sum_{k=0}^{K-1}T^k$ satisfy
$(I-T)S_K = I - T^K$. Since $\|T^K\|\le \|T\|^K\to 0$, $S_K$ converges in operator norm to a bounded operator $S$
and $(I-T)S=I$. The tail bound follows from
\[
  (I-T)^{-1}-S_K = \sum_{k=K}^\infty T^k,
  \qquad
  \Big\|\sum_{k=K}^\infty T^k\Big\|
  \le \sum_{k=K}^\infty \|T\|^k
  = \frac{\|T\|^K}{1-\|T\|}.
\]
\end{proof}

\paragraph{Neumann expansion of the resolvent.}
For $z\in\mathcal{D}$ we write
\begin{equation}\label{eq:resolvent-factor}
  L_n - zI \;=\; (1-z)I - W_n \;=\; (1-z)\Big(I - \frac{W_n}{1-z}\Big).
\end{equation}
Set $T_{n}(z)\coloneqq W_n/(1-z)$. Using \eqref{eq:Wn-norm}, for $z\in\mathcal{D}$ we have
\[
  \|T_{n}(z)\| \le \frac{\|W_n\|}{|1-z|} \le \frac{1}{|1-z|} < 1.
\]
Applying Lemma~\ref{lem:neumann-general} yields the following expansion.

\begin{proposition}[Neumann expansion with uniform tail control]\label{prop:neumann}
For every $z\in\mathcal{D}$,
\begin{equation}\label{eq:neumann-resolvent}
  G_n(z)=(L_n-zI)^{-1}
  =\frac{1}{1-z}\sum_{k=0}^{\infty}\Big(\frac{W_n}{1-z}\Big)^k,
\end{equation}
where the series converges in operator norm. Moreover, for every integer $K\ge 1$,
\begin{equation}\label{eq:neumann-resolvent-tail}
  \Big\|G_n(z)-\frac{1}{1-z}\sum_{k=0}^{K-1}\Big(\frac{W_n}{1-z}\Big)^k\Big\|
  \le \frac{|1-z|^{-K}}{|1-z|-1}.
\end{equation}
\end{proposition}

\begin{proof}
Combine \eqref{eq:resolvent-factor} with Lemma~\ref{lem:neumann-general} for $T=T_n(z)$ and note that
\[
  \|(I-T_n(z))^{-1}-\sum_{k=0}^{K-1}T_n(z)^k\|
  \le \frac{\|T_n(z)\|^K}{1-\|T_n(z)\|}
  \le \frac{|1-z|^{-K}}{1-|1-z|^{-1}}
  =\frac{|1-z|^{-K}}{|1-z|-1}.
\]
Multiplying by $|1-z|^{-1}$ in \eqref{eq:resolvent-factor} gives \eqref{eq:neumann-resolvent-tail}.
\end{proof}

\begin{corollary}[Reduction to traces of powers]\label{cor:trace-reduction}
For every $z\in\mathcal{D}$ and every integer $K\ge 1$,
\begin{equation}\label{eq:mn-trunc}
  m_n(z)
  =\frac{1}{1-z}\sum_{k=0}^{K-1}\frac{1}{(1-z)^k}\,\frac{1}{n}\mathrm{Tr}(W_n^k)
  \;+\; \varepsilon_{n,K}(z),
\end{equation}
where the truncation error satisfies the uniform bound
\begin{equation}\label{eq:mn-trunc-error}
  |\varepsilon_{n,K}(z)|\le \frac{|1-z|^{-K}}{|1-z|-1}.
\end{equation}
\end{corollary}

\begin{proof}
Take normalized traces in \eqref{eq:neumann-resolvent} and in the truncated version of
\eqref{eq:neumann-resolvent-tail}. Using $|n^{-1}\mathrm{Tr}\,B|\le \|B\|$ for any matrix $B$ yields
\eqref{eq:mn-trunc-error}.
\end{proof}

\section{Random walk representation and locality}\label{sec:locality}

This section establishes the bridge between diagonal powers of the normalized adjacency $W_n$ and
local rooted neighborhoods. The key point is that for fixed $k$, the quantity $(W_n^k)_{uu}$
can be interpreted as a $k$-step return probability of the simple random walk and depends only on
a finite (decorated) neighborhood of $u$.

\subsection{Similarity between \texorpdfstring{$W_n$}{Wn} and the random-walk kernel}\label{subsec:similarity}

Recall that $P_n=D_n^{-1}A_n$ is the transition kernel of the simple random walk on $G_n$ (with multiplicities),
and $W_n=D_n^{-1/2}A_nD_n^{-1/2}$ is the symmetric normalized adjacency.

\begin{proposition}[Random-walk representation of diagonal powers]\label{prop:rw-diagonal}
Let $P_n=D_n^{-1}A_n$ be the random-walk kernel on $G_n$ (with multiplicities) and
$W_n=D_n^{-1/2}A_nD_n^{-1/2}$ the normalized adjacency.
Then, for every integer $k\ge 1$ and every vertex $u\in\{1,\dots,n\}$,
\begin{equation}\label{eq:rw-diagonal}
  W_n \;=\; D_n^{1/2} P_n D_n^{-1/2}
  \qquad\text{and}\qquad
  (W_n^k)_{uu} \;=\; (P_n^k)_{uu}.
\end{equation}
In particular, $(W_n^k)_{uu}$ equals the $k$-step return probability of the simple random walk started at $u$.
\end{proposition}

\begin{proof}
Using $P_n=D_n^{-1}A_n$ we compute
\[
D_n^{1/2} P_n D_n^{-1/2}
= D_n^{1/2}(D_n^{-1}A_n)D_n^{-1/2}
= D_n^{-1/2}A_nD_n^{-1/2}
= W_n.
\]
Raising to the power $k$ gives $W_n^k=D_n^{1/2}P_n^kD_n^{-1/2}$, hence
\[
(W_n^k)_{uu}=(D_n^{1/2})_{uu}(P_n^k)_{uu}(D_n^{-1/2})_{uu}=(P_n^k)_{uu}.
\]
Finally, $(P_n^k)_{uu}=\PP_u(X_k=u\mid G_n)$ is the standard interpretation of $k$-step transition probabilities;
see, e.g., \cite[Ch.~2]{LevinPeres2017}.
\end{proof}
\subsection{Rooted balls and decorated local neighborhoods}\label{subsec:decorated-balls}

Fix a finite multigraph $G$ and a root $u\in V(G)$. For an integer $r\ge 0$, the (rooted) radius-$r$ ball
$B_r(G,u)$ is the rooted multigraph induced by vertices within graph distance at most $r$ from $u$.
To account for the fact that the random-walk transition probabilities involve total degrees in $G$
(and not only degrees within the induced subgraph), we work with a decorated version.

\begin{definition}[Decorated rooted ball]
\label{def:decorated-ball}
Let $G$ be a finite multigraph with adjacency matrix $A$ (with multiplicities) and degree function $d(\cdot)$.
For $r\ge 0$, the \emph{decorated rooted ball} $\widehat B_r(G,u)$ consists of
the rooted induced subgraph $B_r(G,u)$ together with the marks
\begin{enumerate}
\item the edge multiplicities on edges of $B_r(G,u)$ (equivalently, the restriction of $A$ to the vertex set of $B_r(G,u)$),
\item the \emph{full} degrees $\{d(x): x\in V(B_r(G,u))\}$ as vertex marks.
\end{enumerate}
Two decorated rooted balls are identified if there exists a root-preserving isomorphism that preserves
edge multiplicities and vertex degree marks.
\end{definition}

\begin{remark}[Decorations and marked rooted graphs]\label{rem:decorations-marked}
Working with decorated (or \emph{marked}) rooted neighborhoods is standard in the theory of local weak
convergence: one considers rooted graphs equipped with vertex/edge marks 
and compares neighborhoods up to root-preserving isomorphisms that preserve the marks; see, e.g.,
Aldous--Lyons \cite{AldousLyons2007}.
The specific decoration in Definition~\ref{def:decorated-ball} is chosen to match the random-walk kernel
$P(x,y)=A_{xy}/d(x)$: edge multiplicities inside the ball determine the numerators $A_{xy}$, while the
vertex marks $d(x)$ provide the denominators, including contributions of edges leaving the induced subgraph.
With this choice, $k$-step return probabilities are measurable functions of the decorated radius-$k$ ball.
\end{remark}

\subsection{Locality of return probabilities}\label{subsec:return-local}

We write $\widehat B_r$ for a generic radius $r$ and $\widehat B_k$ when the radius is equal to $k$.
\begin{lemma}[Locality]\label{lem:locality}
Fix an integer $k\ge 1$. There exists a deterministic map
\[
  F_k:\{\text{decorated rooted balls of radius }k\}\longrightarrow [0,1]
\]
such that for every $n\ge 2$ and every $u\in\{1,\dots,n\}$,
\begin{equation}\label{eq:locality-Fk}
  (P_n^k)_{uu} \;=\; F_k\big(\widehat B_k(G_n,u)\big).
\end{equation}
In particular, $u\mapsto (W_n^k)_{uu}$ is a bounded local functional of the rooted neighborhood of $u$.
\end{lemma}

\begin{proof}
Fix $k$ and a vertex $u$.
Expanding the matrix product, one can write
\begin{equation}\label{eq:path-expansion}
  (P_n^k)_{uu}
  = \sum_{u=v_0,v_1,\dots,v_{k-1},v_k=u}\ \prod_{t=0}^{k-1} P_n(v_t,v_{t+1}),
\end{equation}
where the sum is over all length-$k$ sequences $(v_0,\dots,v_k)$ of vertices.
If $P_n(v_t,v_{t+1})>0$, then $v_t$ and $v_{t+1}$ are adjacent, so $v_t$ lies at graph distance at most $t$ from $u$.
In particular, every vertex visited by such a path belongs to the ball $B_k(G_n,u)$.
Moreover, each factor $P_n(v_t,v_{t+1})$ equals $(A_n)_{v_t v_{t+1}}/d_n(v_t)$ and therefore depends only on
edge multiplicities between vertices in $B_k(G_n,u)$ and on the full degrees of vertices in the ball.
Hence the right-hand side of \eqref{eq:path-expansion} is a measurable function of the decorated ball
$\widehat B_k(G_n,u)$.
Defining $F_k$ by this expression yields \eqref{eq:locality-Fk}.
Finally, as a return probability, $(P_n^k)_{uu}\in[0,1]$.
\end{proof}

%%%%%%%%%%%%%%%%%%%%%%%%%%%%%%%%%%%%%%%%%%%%%%%%%%%%%%%%%%%%%%%%%%%%%%%%%%%%
\section{Local weak convergence and concentration of local averages}
\label{sec:LWC_concentration}

This section collects the probabilistic input needed to pass from local weak limits
to limits of empirical averages of bounded local functionals. We first recall the
local weak convergence of the Barab\'asi--Albert model towards the P\'olya--point
graph, in the sense of \emph{decorated} rooted balls (Definition~\ref{def:decorated-ball}).
We then prove a concentration (self-averaging) statement for empirical averages of
bounded decorated-local functionals. %As a consequence, two uniformly chosen roots
%are asymptotically independent.

\subsection{One-root local weak convergence for decorated balls}
\label{subsec:one_root_LWC}

Recall the decorated rooted balls $\widehat B_r(G_n,u)$ from Definition~\ref{def:decorated-ball}.
Let $U_n$ be a uniformly chosen vertex of $G_n$, independent of the graph.

\begin{lemma}[Degrees are readable one layer further]
\label{lem:degrees_readable_rplus1}
Fix $r\ge 0$. Let $(G,u)$ be a (multi)graph rooted at $u$, and let $H:=B_{r+1}(G,u)$ be the
(rooted) induced ball of radius $r+1$. Then for every vertex $x$ with $\dist_G(u,x)\le r$,
\[
d_G(x)=d_H(x).
\]
Consequently, the decorated ball $\widehat B_r(G,u)$ is a deterministic function of the
(rooted) ball $B_{r+1}(G,u)$.
\end{lemma}

\begin{proof}
Let $\dist_G(u,x)\le r$ and let $y$ be any neighbor of $x$ in $G$ (counting multiplicity).
Then $\dist_G(u,y)\le \dist_G(u,x)+1\le r+1$, hence $y\in V(H)$. Since $H$ is induced by
$V(H)=\{v:\dist_G(u,v)\le r+1\}$, every edge between $x$ and such a $y$ is present in $H$
(with the same multiplicity). Thus every edge incident to $x$ in $G$ is already present
in $H$, and $d_G(x)=d_H(x)$. The final claim follows because $\widehat B_r(G,u)$ consists
of the induced ball $B_r(G,u)$ together with the degree marks $d_G(x)$ for $\dist(u,x)\le r$,
which are given by $d_H(x)$.
\end{proof}

\begin{theorem}[One-root local weak convergence for decorated balls]
\label{thm:one_root_decorated_LWC}
Fix an integer $r\ge 0$. As $n\to\infty$,
\[
\widehat B_r(G_n,U_n)\ \Rightarrow\ \widehat B_r(G_\infty,o),
\]
where $(G_\infty,o)$ denotes the P\'olya--point graph.
\end{theorem}

\begin{proof}
By \cite{BergerBorgsChayesSaberi2014},
the rooted graphs $(G_n,U_n)$ converge locally weakly to $(G_\infty,o)$ in the usual (non-decorated)
local topology. In particular, for every fixed $R\ge 0$,
\[
B_R(G_n,U_n)\ \Rightarrow\ B_R(G_\infty,o).
\]
By Lemma~\ref{lem:degrees_readable_rplus1}, $\widehat B_r(\cdot)$ is a deterministic function of
$B_{r+1}(\cdot)$. Therefore, applying the continuous mapping theorem to the convergence of
$B_{r+1}(G_n,U_n)$ yields the claimed convergence of $\widehat B_r(G_n,U_n)$.
\end{proof}

\subsection{Concentration of empirical averages of decorated-local functionals}
\label{subsec:concentration_local_averages}

For fixed $r\ge 0$, let $\mathcal X_r$ denote the space of isomorphism classes of decorated rooted
radius-$r$ balls. For a bounded function
$f:\mathcal X_r\to\R$, define the empirical average
\[
S_n(f)\;:=\;\frac1n\sum_{u=1}^n f\!\big(\widehat B_r(G_n,u)\big).
\]

\begin{proposition}[Self-averaging of bounded decorated-local averages]
\label{prop:self_averaging}
Fix $r\ge 0$ and let $f:\mathcal X_r\to\R$ be bounded. Then
\[
S_n(f)\;-\;\E\big[f(\widehat B_r(G_n,U_n))\big]\ \xrightarrow{\PP}\ 0.
\]
\end{proposition}

\begin{proof}
We prove concentration via a truncation + Doob martingale argument.

\smallskip
\noindent\emph{\bf Step 1: truncation.}
For $K\in\N$, define the truncated functional
\[
f^{(K)}(\widehat B_r(G,u))
\;:=\;
f(\widehat B_r(G,u))\,
\mathbf 1\Big\{\max_{x\in B_{r}(G,u)} d_G(x)\le K\Big\}.
\]
 Note that the truncation event
 is measurable with respect to $\widehat B_{r}(G,u)$ since $\widehat B_{r}(G,u)$ records the degree marks $\{d_G(x): x\in V(B_r(G,u))\}$. 
By Lemma~\ref{lem:degrees_readable_rplus1}, the decorated ball $\widehat B_r(G,u)$ is a deterministic function of $B_{r+1}(G,u)$,
so $f^{(K)}(\widehat B_r(G,u))$ is indeed a bounded $(r+1)$-local functional. 
Write $S_n^{(K)}:=S_n(f^{(K)})$.

\smallskip
\noindent\emph{\bf Step 2: bounded martingale increments for the truncated average.}
Let $(\mathcal F_t)_{t\ge 1}$ be the natural filtration of the PA construction ($\mathcal F_t=\sigma(G_t)$),
and define the Doob martingale
\[
M_t^{(K)}\;:=\;\E\!\left[S_n^{(K)}\mid \mathcal F_t\right],\qquad t=1,\dots,n,
\]
so that $M_n^{(K)}=S_n^{(K)}$.

Consider the update at time $t$ (adding vertex $t$ and its $m$ edges).
Only the degrees of the $m$ chosen endpoints and the adjacency relations involving $t$
change at time $t$. Hence, if for some root $u$ the value $f^{(K)}(\widehat B_r(G_n,u))$
changes when the update at time $t$ is modified, then the radius-$(r+1)$ ball around $u$ in $G_n$
must intersect $\{t\}\cup\{W_{t,1},\dots,W_{t,m}\}$, and moreover all degrees in $B_r(G_n,u)$  are $\le K$. Indeed, $f^{(K)}(\widehat B_r(G_n,u))\neq 0$ implies $\max_{x\in B_r(G_n,u)} d_{G_n}(x)\le K$ by definition.

In any graph whose degrees are bounded by $K$, the number of vertices within distance $(r+1)$
of a given vertex is at most
\[
N_{r+1}(K)\;:=\;1+K+K^2+\cdots+K^{r+1}\;\le\;\frac{K^{r+2}-1}{K-1}\,.
\]
Therefore, changing the update at time $t$ can affect the truncated values at \emph{at most}
\[
(m+1)\,N_{r+1}(K)
\]
roots $u$ (those within distance $r+1$ of the modified vertices, inside the region where degrees $\le K$).
Since $f$ is bounded, setting $\|f\|_\infty:=\sup |f|$, we obtain the deterministic Lipschitz bound
\[
\big|S_n^{(K)}(\text{after})-S_n^{(K)}(\text{before})\big|
\;\le\;
\frac{2\|f\|_\infty}{n}\,(m+1)\,N_{r+1}(K).
\]
It follows that the Doob increments satisfy
\[
|M_t^{(K)}-M_{t-1}^{(K)}|
\;\le\;
\frac{2\|f\|_\infty}{n}\,(m+1)\,N_{r+1}(K)
\;=:\;\frac{c_{r,m}(K)\,\|f\|_\infty}{n},
\]
where $c_{r,m}(K)=2(m+1)N_{r+1}(K)$.

By the Azuma--Hoeffding inequality for martingales,
for any $\varepsilon>0$,
\begin{equation}
\label{eq:azuma_truncated}
\PP\Big(|S_n^{(K)}-\E[S_n^{(K)}]|>\varepsilon\Big)
\;\le\;
2\exp\!\left(-\frac{\varepsilon^2\,n}{2\,c_{r,m}(K)^2\,\|f\|_\infty^2}\right).
\end{equation}
In particular, for any fixed $K$, $S_n^{(K)}-\E[S_n^{(K)}]\to 0$ in probability.

\smallskip
\noindent\emph{\bf Step 3: remove the truncation.}
Note that
\[
|S_n(f)-S_n^{(K)}|
\;\le\;
\frac{2\|f\|_\infty}{n}\sum_{u=1}^n
\mathbf 1\Big\{\max_{x\in B_{r}(G_n,u)} d_{G_n}(x)>K\Big\}.
\]
Taking expectations and using the uniform root $U_n$,
\[
\E\big[|S_n(f)-S_n^{(K)}|\big]
\;\le\;
2\|f\|_\infty\,
\PP\Big(\max_{x\in B_{r}(G_n,U_n)} d_{G_n}(x)>K\Big).
\]
By Theorem~\ref{thm:one_root_decorated_LWC} , the right-hand side converges to
$\PP(\max_{x\in B_r(G_\infty,o)} d_{G_\infty}(x)>K)$, which tends to $0$ as $K\to\infty$.

Now choose a sequence $K=K(n)\uparrow\infty$ slowly, e.g.\ $K(n)=\lfloor \log n\rfloor$.
Then \eqref{eq:azuma_truncated} implies $S_n^{(K(n))}-\E[S_n^{(K(n))}]\to 0$ in probability,
and the truncation error $|S_n(f)-S_n^{(K(n))}|$ converges to $0$ in probability by Markov's inequality.
Hence $S_n(f)- \E[f(\widehat B_r(G_n,U_n))]\to 0$ is  claimed.

\end{proof}

%%%%%%%%%%%%%%%%%%%%%%%%%%%%%%%%%%%%%%%%%%%
%------------------------------------Ancienne version-------------------------------------------------------
%%%%%%%%%%%%%%%%%%%%%%%%%%%%%%%%%%%%%%%%%%%%

\subsection{Law of large numbers for decorated-local statistics and application to traces}
\label{subsec:LLN_local_traces}

We now combine the one-root decorated local weak convergence with the self-averaging statement to obtain a law of large numbers for empirical averages of bounded decorated-local functionals. This will be applied to the normalized traces $\frac1n\Tr(W_n^k)$ appearing in the Neumann expansion of the resolvent.

\begin{corollary}[LLN for bounded decorated-local functionals]
\label{cor:LLN_local}
Fix $r\ge 0$ and let $\Phi:\mathcal X_r\to\R$ be a bounded measurable functional of the decorated
radius-$r$ ball (here $\mathcal X_r$ is the space of decorated rooted radius-$r$ balls).
Then, as $n\to\infty$,
\[
\frac1n\sum_{u=1}^n \Phi\!\big(\widehat B_r(G_n,u)\big)
\ \xrightarrow{\PP}\
\E\big[\Phi\big(\widehat B_r(G_\infty,o)\big)\big].
\]
\end{corollary}

\begin{proof}
Apply Proposition~5.3 with $f=\Phi$ to get
\[
\frac1n\sum_{u=1}^n \Phi\big(\widehat B_r(G_n, u)\big)-\E\big[\Phi\big(\widehat B_r(G_n,U_n)\big)\big]\xrightarrow{\PP}0.
\]
By Theorem~5.2 and boundedness of $\Phi$,
$\E[\Phi(\widehat B_r(G_n,U_n))]\to \E[\Phi(\widehat B_r(G_\infty,o))]$.
Combine the two displays.
\end{proof}

\subsection{Application to traces of powers}\label{subsec:LLN-traces}

\begin{corollary}[Convergence of normalized traces of powers]\label{cor:trace-powers}
Fix an integer $k\ge 1$. Then
\begin{equation}\label{eq:trace-power-limit}
  \frac{1}{n}\mathrm{Tr}(W_n^k)\ \xrightarrow{\PP}\ \E\big[(W_\infty^k)_{oo}\big]
  \;=\;\E\big[(P_\infty^k)_{oo}\big],
\end{equation}
where $P_\infty=D_\infty^{-1}A_\infty$ is the random-walk kernel on $G_\infty$.
\end{corollary}

\begin{proof}
By Lemma~\ref{lem:locality} there exists a bounded measurable $k$-local functional $F_k:X_k\to[0,1]$ such that
$(P_n^k)_{uu}=F_k(\widehat B_k(G_n,u))$ for all $u$.
Using $(W_n^k)_{uu}=(P_n^k)_{uu}=F_k(\widehat B_k(G_n,u))$ and $\Tr(W_n^k)=\sum_{u=1}^n (W_n^k)_{uu}$, we get
\[
\frac1n\Tr(W_n^k)=\frac1n\sum_{u=1}^n F_k\big(\widehat B_k(G_n,u)\big).
\]
Apply Corollary~\ref{cor:LLN_local} with $r=k$ and $\Phi=F_k$.
\end{proof}

\section{Convergence on the Neumann domain}\label{sec:neumann-domain}

In this section we combine the Neumann expansion of Section~\ref{sec:neumann} with the convergence of
local traces from Section~\ref{sec:LWC_concentration} to obtain convergence of the Stieltjes transform $m_n(z)$ on the
domain $\mathcal D=\{z\in\C_+:\ |1-z|>1\}$.

Recall $\mathcal D$ from \eqref{eq:neumann-domain} and the truncation identity
\eqref{eq:mn-trunc}--\eqref{eq:mn-trunc-error} in Corollary~\ref{cor:trace-reduction}.
For $z\in\mathcal D$ define
\begin{equation}\label{eq:m-series}
  m^{(K)}(z)\;\coloneqq\;\frac{1}{1-z}\sum_{k=0}^{K-1}\frac{1}{(1-z)^k}\,
  \E\big[(W_\infty^k)_{oo}\big],
  \qquad
  m_{\mathcal D}(z)\;\coloneqq\;\lim_{K\to\infty} m^{(K)}(z),
\end{equation}
where the limit exists by absolute convergence, since $|(W_\infty^k)_{oo}|\le \|W_\infty\|^k\le 1$ a.s.
and $\sum_{k\ge0}|1-z|^{-k}<\infty$ on $\mathcal D$.

\begin{proposition}[Convergence of $m_n(z)$ on $\mathcal D$]\label{prop:conv-neumann}
For every $z\in\mathcal D$,
\begin{equation}\label{eq:mn-conv-neumann}
  m_n(z)\ \xrightarrow[n\to\infty]{\PP}\ m_{\mathcal D}(z).
\end{equation}
Moreover, on $\mathcal D$ the limit coincides with the candidate transform defined in
\eqref{eq:limit-stieltjes}:
\begin{equation}\label{eq:mD-identification}
  m_{\mathcal D}(z)\;=\;\E\big[g_\infty(o;z)\big],
  \qquad z\in\mathcal D.
\end{equation}
\end{proposition}

\begin{proof}
Fix $z\in\mathcal D$ and let $K\ge 1$.

\smallskip\noindent\textbf{Step 1: truncation of $m_n(z)$.}
By Corollary~\ref{cor:trace-reduction},
\[
  m_n(z)
  = \frac{1}{1-z}\sum_{k=0}^{K-1}\frac{1}{(1-z)^k}\,\frac{1}{n}\Tr(W_n^k)
  + \varepsilon_{n,K}(z),
\qquad
|\varepsilon_{n,K}(z)|\le \frac{|1-z|^{-K}}{|1-z|-1}.
\]
The error bound is deterministic and uniform in $n$.

\smallskip\noindent\textbf{Step 2: limit of the truncated sum as $n\to\infty$.}
For each fixed $k$, Corollary~\ref{cor:trace-powers} yields
$\frac1n\Tr(W_n^k)\xrightarrow{\PP}\E[(W_\infty^k)_{oo}]$.
Since $K$ is fixed, a finite linear combination preserves convergence in probability, hence
\begin{equation}\label{eq:truncated-conv}
  \frac{1}{1-z}\sum_{k=0}^{K-1}\frac{1}{(1-z)^k}\,\frac{1}{n}\Tr(W_n^k)
  \xrightarrow{\PP} m^{(K)}(z).
\end{equation}

\smallskip\noindent\textbf{Step 3: send $K\to\infty$ (tail control).}
Using \eqref{eq:mn-trunc-error} and the triangle inequality, for any $\eta>0$,
\begin{eqnarray*}
 \PP\big(|m_n(z)-m_{\mathcal D}(z)|>3\eta\big)
 \le  
&& \PP\big(|m_n(z)-m_n^{(K)}(z)|>\eta\big) 
+\PP\big(|m_n^{(K)}(z)-m^{(K)}(z)|>\eta\big)\\
&&+\mathbf{1}\big\{|m^{(K)}(z)-m_{\mathcal D}(z)|>\eta\big\},
\end{eqnarray*}
where $m_n^{(K)}(z)$ denotes the truncated sum in \eqref{eq:truncated-conv}.
The first term vanishes for large $K$ because $|m_n(z)-m_n^{(K)}(z)|\le \frac{|1-z|^{-K}}{|1-z|-1}$.
For such fixed $K$, the second term tends to $0$ as $n\to\infty$ by \eqref{eq:truncated-conv}.
Finally, the third term vanishes for large $K$ because $m^{(K)}(z)\to m_{\mathcal D}(z)$ by definition.
This proves \eqref{eq:mn-conv-neumann}.

\smallskip\noindent\textbf{Step 4: identification with $\E[g_\infty(o;z)]$ on $\mathcal D$.}
On the event that $G_\infty$ is locally finite (which holds a.s.), $W_\infty$ is a bounded self-adjoint
operator with $\|W_\infty\|\le 1$ and $L_\infty=I-W_\infty$.
For $z\in\mathcal D$ we factor
\[
  L_\infty-zI = (1-z)\Big(I-\frac{W_\infty}{1-z}\Big),
\qquad
\Big\|\frac{W_\infty}{1-z}\Big\|\le \frac{1}{|1-z|}<1.
\]
Therefore the Neumann series (Lemma~\ref{lem:neumann-general}) applies and yields, in operator norm,
\[
  (L_\infty-zI)^{-1}
  = \frac{1}{1-z}\sum_{k=0}^\infty \Big(\frac{W_\infty}{1-z}\Big)^k.
\]
Taking the $(o,o)$ matrix element gives
\[
  g_\infty(o;z)=\frac{1}{1-z}\sum_{k=0}^\infty \frac{(W_\infty^k)_{oo}}{(1-z)^k},
\qquad z\in\mathcal D,
\]
with an absolutely summable right-hand side since $|(W_\infty^k)_{oo}|\le 1$ a.s.
Hence, by dominated convergence/Fubini,
\[
  \E[g_\infty(o;z)]
  = \frac{1}{1-z}\sum_{k=0}^\infty \frac{\E[(W_\infty^k)_{oo}]}{(1-z)^k}
  = m_{\mathcal D}(z),
\]
which proves \eqref{eq:mD-identification}.
\end{proof}

\begin{remark}\label{rem:start-D}
Locality and local weak convergence only identify $\frac1n\Tr(W_n^k)$ for fixed $k$.
The Neumann expansion on $\mathcal D$ allows us to reduce $m_n(z)$ to finitely many such traces plus
a uniformly small tail, and then send the truncation level $K\to\infty$.
In Section~\ref{sec:extension} we extend the convergence from $\mathcal D$ to all $z\in\C_+$ by an analytic
continuation (normal family/Vitali) argument.
\end{remark}

\section{Extension to all $\C_+$ and convergence of the ESD}\label{sec:extension}

Section~\ref{sec:neumann-domain} establishes convergence of $m_n(z)$ on the nonempty open set
$\mathcal D=\{z\in\C_+:\ |1-z|>1\}$. In this section we extend the convergence to all of $\C_+$ using a
normal-family argument and then deduce weak convergence of the empirical spectral distribution.

\subsection{Normal family}\label{subsec:normal-family}

\begin{lemma}[Local boundedness]\label{lem:local-boundedness}
Let $\K\subset\C_+$ be compact and set $\eta_{\K}\coloneqq \inf\{\Im z:\ z\in \K\}>0$.
Then for all $n\ge 2$,
\begin{equation}\label{eq:uniform-bound-compact}
  \sup_{z\in \K} |m_n(z)| \le \eta_{\K}^{-1}.
\end{equation}
In particular, $(m_n)_{n\ge 2}$ is a locally bounded family of holomorphic functions on $\C_+$.
\end{lemma}

\begin{proof}
Each $m_n$ is holomorphic on $\C_+$ and satisfies $|m_n(z)|\le (\Im z)^{-1}$ by Lemma~\ref{lem:resolvent-bound}.
Taking $z\in \K$ gives $|m_n(z)|\le \eta_{\K}^{-1}$, hence \eqref{eq:uniform-bound-compact}.
\end{proof}

\begin{remark}[Normal families]\label{rem:normal-families}
By Montel's theorem, any locally bounded family of holomorphic functions on a domain is relatively compact
for the topology of uniform convergence on compacts; see, e.g., \cite[Ch.~10]{ConwayComplexI}.
Lemma~\ref{lem:local-boundedness} therefore implies that $(m_n)$ is a normal family on $\C_+$.
\end{remark}

\subsection{Vitali extension from $\mathcal D$ to $\C_+$}\label{subsec:vitali}

Let $m(z)\coloneqq \E[g_\infty(o;z)]$, $z\in\C_+$, as in Section~\ref{subsec:lwc-limit-operator}.
By Proposition~\ref{prop:conv-neumann}, we already know that $m_n(z)\to m(z)$ in probability for every $z\in\mathcal D$,
and that $m_{\mathcal D}(z)=m(z)$ on $\mathcal D$.

\begin{proposition}[Extension to all $z\in\C_+$]\label{prop:vitali}
For every $z\in\C_+$,
\begin{equation}\label{eq:mn-conv-all}
  m_n(z)\ \xrightarrow[n\to\infty]{\PP}\ m(z).
\end{equation}
Moreover, for every compact $\K\subset\C_+$, the convergence holds uniformly on $\K$ in probability:
\begin{equation}\label{eq:mn-unif-compact}
  \sup_{z\in \K} |m_n(z)-m(z)|\ \xrightarrow[n\to\infty]{\PP}\ 0.
\end{equation}
\end{proposition}

\begin{proof}
We use a subsequence principle together with a deterministic Vitali theorem.

\smallskip\noindent\textbf{Step 1: almost sure convergence on a countable set along a subsubsequence.}
Let $(n_j)$ be an arbitrary subsequence. Fix a countable dense set $\{z_q\}_{q\ge 1}\subset\mathcal D$.
By Proposition~\ref{prop:conv-neumann}, for each fixed $q$ we have $m_{n_j}(z_q)\to m(z_q)$ in probability.
Hence, for each $q$, there exists a further subsequence $(n_j^{(q)})$ along which
$m_{n_j^{(q)}}(z_q)\to m(z_q)$ almost surely.
By a diagonal extraction, we can find a subsubsequence $(n_{j_\ell})$ such that
\begin{equation}\label{eq:as-conv-countable}
  m_{n_{j_\ell}}(z_q)\ \longrightarrow\ m(z_q)\qquad\text{a.s.\ for every }q\ge 1.
\end{equation}

\smallskip\noindent\textbf{Step 2: apply Vitali's theorem pathwise.}
Fix $\omega$ in the probability-one event where \eqref{eq:as-conv-countable} holds.
For this $\omega$, the functions $z\mapsto m_{n_{j_\ell}}(z,\omega)$ are holomorphic on $\C_+$ and locally bounded
by Lemma~\ref{lem:local-boundedness}. Moreover, they converge pointwise on the set $\{z_q\}$, which has an
accumulation point in $\C_+$.
By Vitali's theorem (or Montel + uniqueness of analytic continuation), the sequence
$m_{n_{j_\ell}}(\cdot,\omega)$ converges uniformly on compacts in $\C_+$ to a holomorphic limit $h(\cdot,\omega)$;
see, e.g., \cite[Ch.~10]{ConwayComplexI}. By \eqref{eq:as-conv-countable} and continuity, $h(z,\omega)=m(z)$ for all
$z\in\mathcal D$, and hence by analytic continuation $h(z,\omega)=m(z)$ for all $z\in\C_+$.

Therefore, along the subsubsequence $(n_{j_\ell})$ we have
\begin{equation}\label{eq:as-unif-compact}
  \sup_{z\in K}|m_{n_{j_\ell}}(z)-m(z)|\ \longrightarrow\ 0
  \qquad\text{a.s.\ for every compact }K\subset\C_+.
\end{equation}

\smallskip\noindent\textbf{Step 3: conclude convergence in probability for the full sequence.}
We have shown that \emph{every} subsequence $(n_j)$ admits a subsubsequence $(n_{j_\ell})$ such that
\eqref{eq:as-unif-compact} holds. Since the limit $m$ is deterministic, this implies
\eqref{eq:mn-unif-compact} for the whole sequence (a standard subsequence principle).
In particular, \eqref{eq:mn-conv-all} holds for each fixed $z\in\C_+$.
\end{proof}

\subsection{From Stieltjes transforms to measures}\label{subsec:stieltjes-to-measure}

Let $\mu_n$ be the empirical spectral distribution of $L_n$ and recall $m_n(z)=\int (x-z)^{-1}\,\mu_n(dx)$,
$z\in\C_+$.

\begin{corollary}[Weak convergence of the ESD]\label{cor:esd}
There exists a deterministic probability measure $\mu$ on $\R$ such that, as $n\to\infty$,
\[
  \mu_n\ \xRightarrow{\PP}\ \mu.
\]
Its Stieltjes transform is $m(z)=\int_\R (x-z)^{-1}\,\mu(dx)$, $z\in\C_+$, and $m(z)=\E[g_\infty(o;z)]$.
\end{corollary}

\begin{proof}
The function $m$ is a Herglotz function on $\C_+$ (as an expectation of Green functions) and therefore is the
Stieltjes transform of a unique probability measure $\mu$ on $\R$ (see, e.g., \cite[Ch.~VII]{ReedSimon1980}).
By Proposition~\ref{prop:vitali}, $m_n(z)\to m(z)$ in probability for every $z\in\C_+$.
The Stieltjes continuity theorem then implies $\mu_n\Rightarrow \mu$ in probability; see, e.g.,
\cite[Thm.~2.4]{BaiSilverstein2010} 
\end{proof}

\begin{remark}[Support]\label{rem:support}
Since $\mathrm{Spec}(L_n)\subset[0,2]$ for every $n$, the measures $\mu_n$ are supported on $[0,2]$, and so is $\mu$.
\end{remark}

%%%%%%%%%%%%%%%%%%%%%%%%%%%%%%%%%%%%%%%%%%%%%%%%%%%%%%%%%%%%%%%%%%%
%%                                                               %%
%% You may add acknowledgments (optional).                       %%
%%                                                               %%
%%%%%%%%%%%%%%%%%%%%%%%%%%%%%%%%%%%%%%%%%%%%%%%%%%%%%%%%%%%%%%%%%%%
%\begin{acks}
%We are grateful to Martin Hairer who provided a nice \texttt{MR} macro and to S\'ebastien Gou\"ezel for his useful comments on the internals of the class file.
%\end{acks}

%%%%%%%%%%%%%%%%%%%%%%%%%%%%%%%%%%%%%%%%%%%%%%%%%%%%%%%%%%%%%%%%%%%
%%                                                               %%
%% You have reached the end of your document.                    %%
%%                                                               %%
%%%%%%%%%%%%%%%%%%%%%%%%%%%%%%%%%%%%%%%%%%%%%%%%%%%%%%%%%%%%%%%%%%%

\end{document}